\documentclass[11pt,leqno,oneside,letterpaper]{amsart}
\usepackage{amssymb,amsmath,amsfonts,latexsym, amscd, amsthm, graphicx, url, epsfig, mathtools, algorithm, algpseudocode}
\usepackage[letterpaper,left=1.75truein, top=1truein]{geometry}
\usepackage{color}
\usepackage[colorlinks,linkcolor=blue,citecolor=blue]{hyperref}
\newtheorem{theorem}{Theorem}[section]
\newtheorem{lemma}[theorem]{Lemma}
\newtheorem{proposition}[theorem]{Proposition}

\newtheorem{corollary}[theorem]{Corollary} 

\theoremstyle{definition}

\newtheorem{example}[theorem]{Example}

\newtheoremstyle{cases}
  {12pt plus 6 pt}
  {2pt}
  {\bfseries}   
  {}
  {\bfseries}
  {.}
  {.5em}
  {}
\theoremstyle{cases}

\numberwithin{subcase}{case} \numberwithin{subsubcase}{subcase}
\numberwithin{equation}{subsection}

\begin{document}

\title[Actions on $\mathrm{LO}(G)$]{Automorphisms acting on the left-orderings of a bi-orderable group }

\author[Adam Clay and Sina Zabanfahm]{Adam Clay and Sina Zabanfahm}
\thanks{Adam Clay was partially supported by NSERC grant RGPIN-2014-05465}
\address{Department of Mathematics, 420 Machray Hall, University of Manitoba, Winnipeg, MB, R3T 2N2.}
\email{Adam.Clay@umanitoba.ca, zabanfas@umanitoba.ca}
\urladdr{http://server.math.umanitoba.ca/~claya/}

\begin{abstract}
We generalize a result of Koberda \cite{Koberda11}, by showing that the natural action of the automorphism group on the space of left-orderings is faithful for all nonabelian bi-orderable groups $G$, as well as for a certain class of left-orderable groups that includes the braid groups.   As a corollary we show that the action of $\mathrm{Aut}(G)$ on $\partial G$ is faithful whenever $G$ is bi-orderable and hyperbolic, following the approach of \cite{Koberda11}.   We also analyze the action of the commensurator of $G$ on its space of virtual left-orderings. 
\end{abstract}
\maketitle
\vspace{-.6cm}
\begin{center}
\today
\end{center}

\section{Introduction}
Let $G$ be a group.  We call a strict total ordering $<$ of the elements of $G$ a \textit{left-ordering} if $g<h$ implies $fg<fh$ for all $f, g,h \in G$.   If $G$ admits a left-ordering $<$ that is also right-invariant, in the sense that $g<h$ implies $gf<hf$ for all $f,g,h \in G$, then $<$ is a \textit{bi-ordering} of $G$. 

Each of these concepts can equivalently be defined in terms of positive cones.  That is, given a left-ordering $<$ of $G$, we can identify $<$ with its positive cone 
\[ P =\{ g \in G \mid g>1 \}
\]
which is a subset of $G$ satisfying:
\begin{enumerate}
\item $P \cdot P \subset P$
\item $P \sqcup P^{-1} \sqcup \{ 1 \} = G$.
\end{enumerate}
Conversely, given a subset $P \subset G$ satisfying $(1)$ and $(2)$, it determines a positive cone according to the prescription $g<h$ if and only if $g^{-1}h \in P$ for all $g, h \in G$.  Bi-orderings may be similarly defined in terms of positive cones, but the positive cone of any bi-ordering must also satisfy a third condition, namely $gPg^{-1} \subset P$ for all $g \in G$.

We write $\mathrm{LO}(G)$ for the set of all positive cones $P \subset G$ satisfying $(i)$ and $(ii)$ above, and, thinking of it as a subset of $2^G$ (equipped with the product topology) we endow $\mathrm{LO}(G)$ with the subspace topology.  Thus the open sets of $\mathrm{LO}(G)$ are finite intersections of sets of the form 
\[ U_g = \{ P \in \mathrm{LO}(G) \mid g \in P \} \mbox{ and } U_g^c = \{ P \in \mathrm{LO}(G) \mid g^{-1} \in P \}.
\]
We call $\mathrm{LO}(G)$ the space of left-orderings of the group $G$.  We similarly can define the space of bi-orderings of $G$, $\mathrm{BiO}(G)$, by taking all positive cones $P$ that satisfy the additional third condition of $gPg^{-1} \subset P$ for all $g \in G$.  Topologizing $\mathrm{BiO}(G)$ in the same way, we evidently have $\mathrm{BiO}(G) \subset \mathrm{LO}(G)$. Endowed with these topologies, both $\mathrm{LO}(G)$ and $\mathrm{BiO}(G)$ are compact spaces.  

There is an action of $G$ on $\mathrm{LO}(G)$ defined by $g(P) = gPg^{-1}$.  More generally, there is an action of $\mathrm{Aut}(G)$ on $\mathrm{LO}(G)$ by observing that $\phi(P)$ is again a positive cone for all $P \in \mathrm{LO}(G)$ and $\phi \in \mathrm{Aut}(G)$.  The action of $\mathrm{Aut}(G)$ on $\mathrm{LO}(G)$ is an action by homeomorphisms.   Since the positive cones which are fixed under conjugation correspond to the bi-orderings of $G$, there is also an action of $\mathrm{Out}(G)$ on $\mathrm{BiO}(G)$.

With the topological structure and group actions as above, $\mathrm{LO}(G)$ has found many applications within the study of orderable groups (for example, it was used to show that every left-orderable group has finitely many or uncountably many left-orderings \cite{linnell11}, and was used to demonstrate a connection between orderability and amenability \cite{Morris06}), though applications beyond the realm of orderability are few.   In recent work Koberda provided an example of such an application, by showing that whenever $G$ is a residually torsion-free nilpotent hyperbolic group, the natural action of $\mathrm{Aut}(G)$ on $\partial G$ is faithful \cite{Koberda11}.  This application relies on the following theorem, which was also extended in \cite{Morris12} by replacing $\mathrm{Aut}(G)$ with the commensurator of $G$:

\begin{theorem} \cite[Theorem 1.1]{Koberda11}
\label{koberdas theorem}
If $G$ is a finitely generated residually torsion-free nilpotent group, then the natural action of $\mathrm{Aut}(G)$ on $\mathrm{LO}(G)$ is faithful.
\end{theorem}
 
In this paper, we characterize the action of $\mathrm{Aut}(G)$ on $\mathrm{LO}(G)$ when $G$ is a bi-orderable group.   Recall that finitely-generated residually torsion-free nilpotent groups are bi-orderable, though the converse is not true.  For example, Thompson's group $F$ is bi-orderable, but not residually nilpotent since $[F,F]$ is a simple group \cite[Section 1.2.4]{DNR14}.   

Note that for some bi-orderable groups, like $\mathbb{Q}^k$ for all $k>0$, we should not expect the action of $\mathrm{Aut}(G)$ on $\mathrm{LO}(G)$ to be faithful.  For if $G= \mathbb{Q}^k$ then  multiplication by a positive rational $p/q$ in each coordinate of  $\mathbb{Q}^k$ can easily be seen to preserve all orderings of $\mathbb{Q}^k$.   However, it turns out that these automorphisms of abelian groups are the only nontrivial automorphisms of bi-orderable groups which act trivially on the space of left-orderings.   For an abelian group $G$ and a fixed $p/q \in \mathbb{Q}$, we denote by $\tau_{p/q} : G \rightarrow G$ the automorphism satisfying $\tau_{p/q}(g^q) = g^p$ for all $g \in G$, when it exists. We prove:

\begin{theorem}
\label{main theorem}
Let $G$ be a bi-orderable group.

\noindent (i) If $G$ is nonabelian then $\mathrm{Aut}(G)$ acts faithfully on $\mathrm{LO}(G)$.

\noindent (ii) If $G$ is abelian then the kernel of the action of $\mathrm{Aut}(G)$ on $\mathrm{LO}(G)$ contains precisely the automorphisms $\tau_{p/q}$, if any such automorphisms exist.

\end{theorem}

Note that part (ii) of Theorem \ref{main theorem} already appears as \cite[Proposition 4.3(2)]{Morris12}.   We are also able to analyze the behaviour of the action of $\mathrm{Aut}(G)$ on $\mathrm{LO}(G)$ with respect to certain kinds of extensions.  

\begin{theorem}
\label{ses theorem}
Suppose that $G$ is left-orderable and that
\[ 1 \rightarrow K \rightarrow G \rightarrow \mathbb{Z} \rightarrow 1
\]
is a short exact sequence of groups.  Suppose that $\mathrm{Aut}(K)$ acts faithfully on $\mathrm{LO}(K)$.  If conjugation by the generator of $\mathbb{Z}$ preserves a left-ordering of $K$, then $\mathrm{Aut}(G)$ acts faithfully on $\mathrm{LO}(G)$.
\end{theorem}

Since bi-orderability is not preserved under extensions (even under extensions such as those in the statement of the theorem above), this allows us to create non-bi-orderable groups $G$ for which $\mathrm{Aut}(G)$ acts faithfully on $\mathrm{LO}(G)$.  See also Proposition \ref{smallest element proposition}.

As a corollary of Theorem \ref{main theorem} we can extend Koberda's result concerning the action of $\mathrm{Aut}(G)$ on $\partial G$ to all bi-orderable hyperbolic groups.

\begin{corollary}
\label{hyperbolic}
If $G$ is a bi-orderable hyperbolic group, then $\mathrm{Aut}(G)$ acts faithfully on $\partial G$.
\end{corollary}

The proof of Corollary \ref{hyperbolic} is a combination of Theorem \ref{main theorem} and Proposition \ref{Koberda prop}.

The paper is organized as follows.  In Section \ref{main section} we provide additional background on left-orderings and bi-orderings of groups, and prove Theorem \ref{main theorem}.  In Section \ref{nonbo groups} we prove Theorem \ref{ses theorem} and also study the braid groups $B_n$.  In Section \ref{VLOG} show that the action of $\mathrm{Aut}(G)$ on $\partial G$ is faithful when $G$ is hyperbolic and bi-orderable, and describe the action of $\mathrm{Comm}(G)$ on $\mathrm{VLO}(G)$ for all bi-orderable groups.

\section{Automorphisms of bi-orderable groups acting on the space of orderings }
\label{main section}

By insisting that the group $G$ be bi-orderable, we allow ourselves some flexibility in creating new left-orderings of $G$.  The orderings that we will create arise from considering the action of $G$ on itself by conjugation, which is an order-preserving action if $G$ is bi-ordered (see Lemma \ref{new order}).  With this line of reasoning we will create sufficiently many left-orderings to show that whenever $\phi \in \mathrm{Aut}(G)$ and $\phi(g) \neq g$ for some $g \in G$, then there exists $P \in \mathrm{LO}(G)$ that contains $g$ but not $\phi(g)$.  It follows that the action of $\phi$ on $\mathrm{LO}(G)$ is nontrivial, because the positive cone $P$ satisfies $\phi(P) \neq P$.

Recall that a subset $S \subset G$ is called \textit{isolated} if $g^k \in S$ for some $k \in \mathbb{Z}$ implies that $g \in S$.  The \textit{isolator} of a subgroup $H$ of $G$ is the set 
\[ I(H) = \{ g \in G  \mid \mbox{ there exists $k \in \mathbb{Z}$ such that $g^k \in H$} \}.
\]
In general, $I(H)$ is not a subgroup.  However, when $H$ is abelian and $G$ is bi-orderable, then $I(H)$ is an abelian subgroup.  Essential in proving this fact is the following property of bi-orderable groups:  In a bi-orderable group, when $g^k$ and $h^{\ell}$ commute for some $k, \ell \in \mathbb{Z}$, then so do $g$ and $h$.  This fact will also be used several times in the proofs of this section.  

When $H$ is a rank one abelian subgroup of $G$, so is $I(H)$.  If $g$ is a nonidentity element of a bi-orderable group $G$, then we will denote the isolator of the cyclic subgroup $\langle g \rangle$ by $I(g)$ for short. Thus $I(g)$ is always a rank one abelian group.  We record the following fact for future use:

\begin{lemma}
\label{Ig lemma}
Let $G$ be a group. If $g, h$ are distinct elements of $G$, then either $I(g) = I(h)$ or $I(g) \cap I(h) = \{ 1 \}$.
\end{lemma}
\begin{proof}
Suppose there exists $f \in I(h) \cap I(g)$ where $f\not = 1$. Since $f\in I(g)$, there exist $n,m \in \mathbb{Z}$ such that $f^n=g^m$. But now $g^n \in I(h)$ and since $I(h)$ is isolated, $g$ is also in $I(h)$ and $I(h)=I(g)$.
\end{proof}

Recall that a subset $S$ in a left-ordered group $G$ is called \textit{convex} with respect to a given left-ordering $<$ if  $g, h \in S$ and $g<f<h$ implies $f \in S$.   Of particular importance is the case when a subgroup $C$ of a left-ordered group $G$ is convex, as the convex subgroups of a left-ordering determine its structure in a sense described below.  The convex subgroups of a left-ordered group $G$ are ordered by inclusion.  A subgroup is \textit{relatively convex} if there exists a left-ordering relative to which it is convex.  

Given a subgroup $C$ of a left-ordered group $G$, the natural quotient ordering of the left cosets $G/C$ is well-defined if and only if $C$ is convex, in this case the natural left-action of $G/C$ preserves the quotient ordering.  Therefore we can think of the ordering of $G$ as lexicographic: it is constructed via inclusion of the left-ordered subgroup $C$ and via pullback of the natural ordering on the cosets $G/C$.

Consequently, if $C$ is a convex subgroup of a left-ordered group, then the left-ordering of $G$ may be altered by replacing the left-ordering of $C$ with any left-ordering that we please.  It follows that relative convexity is transitive, in the sense that if $K$ is relatively convex in $H$, and $H$ is relatively convex in $G$, then $K$ is relatively convex in $G$.  This fact is needed in the proof of the following lemma.

\begin{lemma} \cite[Lemma 2.4]{Clay12}
\label{new order}
Suppose that $G$ is a bi-orderable group, and that $g \in G$ is not the identity.  Then $I(g)$ is relatively convex.
\end{lemma}
\begin{proof}  
Let $G_i$, $i=1,2$ denote two copies of the group $G$, and equip each copy with a given bi-ordering $<$.  Create a total ordering of $G_1 \cup G_2$ using $<$ to order each $G_i$, and declare the elements of $G_1$ smaller than those of $G_2$. 

Now consider the action of $G$ on $G_1 \cup G_2$ defined by conjugation on the elements of $G_1$, and by left-multiplication on the elements of $G_2$.  This defines an effective, order-preserving action of $G$ on the totally ordered set $G_1 \cup G_2$.  Fix a nonidentity element $g \in G_1$ and well-order $G_1 \cup G_2$ so that $g$ is smallest.  Then using the action of $G$ on $G_1 \cup G_2$ one may create a left-ordering of $G$ in the standard way, relative to which $Stab_G(g) = C_G(g)$ is convex.  Here, $C_G(g)$ denotes the centralizer of $g$ in $G$ (See \cite[Proposition 2.3]{Clay12} or \cite[Example 1.11 and Problem 2.16]{CR15} for details of this construction).  Now as $C_G(g)$ is bi-orderable, the centre $Z(C_G(g))$ is relatively convex in $C_G(g)$ by \cite[Theorem 2.4]{MR77}. Moreover, $I(g) \subset Z(C_G(g))$ since every element of $I(g)$ has some power which lies in $\langle g \rangle$, and thus commutes with all elements of $C_G(g)$.  Since $I(g)$ is an isolated subgroup and $Z(C_G(g))$ is abelian, $I(g)$ is relatively convex in $Z(C_G(g))$.  Thus $I(g)$ is relatively convex in $G$.
\end{proof}

\begin{proposition}
\label{main lemma}
Suppose that $G$ is a bi-orderable group, and that $\phi \in \mathrm{Aut}(G)$.  If there exists $g \in G$ such that $\phi(I(g)) \neq I(g)$, or if there exists $g \in G$ such that $\phi(g)^n = g^{-m}$ for some $m, n >0$, then the action of $\phi$ on $\mathrm{LO}(G)$ is nontrivial.
\end{proposition}
\begin{proof} 
Suppose there exists $g \in G$ such that $\phi(g)^n=g^{-m}$ for some $m,n > 0$. Consider an arbitrary positive cone $P \in  LO(G)$. We can assume $g \in P$, if not we replace $P$ by $P^{-1}$. Then $g\in P$ and $\phi(g) \not \in P$, so we have $P \not = \phi(P)$.

Now suppose there exists $g$ such that $I(g)\not=\phi(I(g)) $, and note that $\phi(I(g)) = I(\phi(g))$. By Lemma \ref{new order} $I(g)$ is convex in some left-ordering of $G$ with positive cone $P \in LO(G)$.  Applying $\phi$, one checks that $\phi(I(g)) = I(\phi(g))$ is convex relative to the ordering of $G$ determined by $\phi(P)$.  

To show that $\phi(P) \neq P$, we need only show that $I(g)$ is not convex relative to the ordering of $G$ determined by $\phi(P)$.  If it were, we would have either $I(g) \subset I(\phi(g))$ or $I(\phi(g)) \subset I(g)$, since convex subgroups are ordered by inclusion. By Lemma \ref{Ig lemma}, either inclusion forces $I(g) = I(\phi(g)) = \phi(I(g))$, a contradiction.
\end{proof}

Therefore, by Proposition \ref{main lemma}, when $G$ is a bi-orderable group and $\phi \in \mathrm{Aut}(G)$ we know that $\phi$ acts nontrivially on $\mathrm{LO}(G)$ unless $\phi$ satisfies:
\[ \label{autom} \tag{*} \forall g \in \mathrm{domain}(\phi) \ \exists n,m >0 \mbox{ such that } \phi(g)^n = g^m.
\]
We therefore investigate the existence of such automorphisms of bi-orderable groups.  

Our lemmas below are stated in a slightly more general setting than needed in this section, as we will also be using them in our investigation of the action of $\mathrm{Comm}(G)$ on $\mathrm{VLO}(G)$ in Section \ref{VLOG}.

Recall that when $G$ is abelian, we denote by $\tau_{p/q} : G \rightarrow G$ the automorphism satisfying $\tau_{p/q}(g^q) = g^p$ for all $g \in G$, when it exists.  More generally, if $H_1, H_2$ are finite index abelian subgroups of a group $G$, we denote by $\tau_{p/q} : H_1 \rightarrow H_2$ the isomorphism satisfying $\tau_{p/q}(g^q) = g^p$ for all $g \in H_1$, when it exists.

\begin{lemma}
\label{Morris-lemma}
Suppose $G$ is a bi-orderable group with finite index torsion-free abelian subgroups $H_1, H_2$, and $\phi:H_1 \rightarrow H_2$ is an isomorphism satisfying (\ref{autom}).  Then there exist $p,q >0$ such that $\phi(g)^q = g^p$ for all $g \in H_1$, so that $\phi = \tau_{p/q}$.
\end{lemma}
\begin{proof} This lemma is essentially Case 2 of the proof of \cite[Proposition 4.3]{Morris12}.  Here is an alternative proof.  Assume $\phi:H_1 \rightarrow H_2$ satisfies (\ref{autom}) and that $H_1$ is torsion free abelian.  Let $g,h \in H_1$ and suppose $\phi(g)^m = g^n$ and $\phi(h)^{\ell} = h^k$ for some $k,\ell,m,n >0$.   By uniqueness of roots, we may assume that $\gcd(m,n)=\gcd(k, \ell) =1$, we wish to show that $m=\ell$ and $n=k$.  If $I(g) = I(h)$ then the result follows by applying $\phi$ to a common power of $g$ and $h$ which lies in $H_1$, such a common power exists since $|G:H_1|$ is finite.  So suppose $I(g) \neq I(h)$, and therefore $I(g) \cap I(h) = \{ 1 \}$ by Lemma \ref{Ig lemma}. 

Considering $g^m h^{\ell}$, we see that $\phi(g^m h^{\ell}) = g^nh^k \in I(g^m h^{\ell})$, so there exist relatively prime $s, t >0$ such that $(g^m h^{\ell})^s = (g^nh^k)^t$.  Since $H_1$ is abelian $g^{ms-nt} = h^{tk-s\ell}$, and since both are in $I(g) \cap I(h)$, both are equal to $1$.  Since $\gcd(m,n)= \gcd(s,t)=1$, from $ms-nt=0$ we find $m=t$ and $s=n$.  Similarly from $tk-s\ell$ we find $t= \ell$ and $k=s$, so we are done.
\end{proof}

\begin{lemma} 
\label{inequality}
Suppose $G$ is a bi-orderable group with finite index subgroups $H_1, H_2$, that $\phi:H_1 \rightarrow H_2$ is an isomorphism satisfying (\ref{autom}), and that $\phi$ is not the identity.  Then for every $g \in H_1$ there exist $p, q >0$ such that $\phi(g)^q = g^p$ where $p \neq q$.
\end{lemma}
\begin{proof}
Since $\phi$ is not the identity there exists $g\in H_1$ with $\phi(g) \neq g$, say $\phi(g)^s=g^t$ with $s \neq t$ (necessarily $s \neq t$ since $G$ is bi-orderable). Now let $h\in G$ be given. By (\ref{autom}) there exists $n,m >0$ such that $\phi(h)^n=h^m$. If $n=m$ then $\phi(h)=h$ since $G$ is bi-orderable. But $\phi(g^sh)=g^th$, so $g^th\in I(g^sh).$  But then $(g^th)(h^{-1}g^{-s})=g^{t-s}\in I(g^sh)$. Therefore $g \in I(g^sh)$, and so $h\in I(g^sh)$, and $I(g)=I(h)$. Now since $I(g)$ is abelian we may apply Lemma \ref{Morris-lemma} to the restriction isomorphism $\phi|_{I(g)}: I(g) \rightarrow I(g)$ arising from $\phi$.  We conclude that $n=s$ and $m=t$, contradicting the fact that $n=m$. Thus $n \neq m$.
\end{proof}

Note that we can improve the conclusion of the previous lemma, by using uniqueness of roots in a bi-orderable group to show that $p,q$ exist with $\gcd(p,q) =1$.  However this is not needed for our purposes.

\begin{lemma} 
\label{funny element}
Suppose $G$ is a bi-orderable group with finite index subgroups $H_1, H_2$ and that $\phi:H_1 \rightarrow H_2$ is an isomorphism satisfying (\ref{autom}).  Let $g, h \in H_1$ be given and suppose that $\phi(g)^m = g^n$ and $\phi(h)^{\ell}=h^k$.  Then 
\[g^{n-m}hg^{m-n} \in I(h) \mbox{ and } h^{k-\ell}gh^{\ell-k} \in I(g).
\]
\end{lemma}
\begin{proof}
By symmetry, it suffices to show only $g^{n-m}hg^{m-n}\in I(h)$. First, notice that $\phi(f) \in I(f)$ for all $f\in H_1$ by (\ref{autom}).  Therefore $\phi(g^mh^\ell g^{-m})=g^nh^kg^{-n} \in I(g^mhg^{-m})$, and since $I(g^mhg^{-m})$ is isolated we conclude $g^nhg^{-n}\in I(g^mhg^{-m})$.  Next, notice that if $x\in I(h)$ then $g^ixg^{-i}\in I(g^ihg^{-i})$ for all $ i\in \mathbb{Z}$, and thus $g^{n-m}hg^{m-n}\in I(h)$.
\end{proof}

\begin{lemma}
\label{no star phis}
Suppose $G$ is a bi-orderable group with finite index subgroups $H_1, H_2$, that $\phi:H_1 \rightarrow H_2$ is an isomorphism satisfying (\ref{autom}), and that $\phi$ is not the identity.  Then $H_1$ is abelian.
\end{lemma}
\begin{proof}
Let $g,h \in H_1$ be given. If $I(g)=I(h)$ then $g$ and $h$ commute. Thus we assume $I(g) \neq I(h)$.  By Lemma \ref{inequality} there exist $m,n >0$ and $k, \ell >0$ with $m\neq n$ and $k\neq \ell$ such that $\phi(g)^m=g^n$ and $\phi(h)^\ell = h^k$. Consider $h^{k-\ell}g^{n-m}h^{\ell-k}g^{m-n}$. On one hand, we have $h^{k-\ell}g^{n-m}h^{\ell-k}g^{m-n}=(h^{k-\ell}g^{n-m}h^{\ell-k})\cdot g^{m-n} \in I(g)$, since it is a product of elements of $I(g)$ (here we use Lemma \ref{funny element}). On the other hand, $h^{k-\ell}\cdot(g^{n-m}h^{\ell-k}g^{m-n})\in I(h)$ by similar reasoning. By Lemma \ref{Ig lemma} $I(g)\cap I(h)=\{1\}$ and so $h^{k-\ell}g^{n-m}h^{\ell-k}g^{m-n}$=1. But this means the nontrivial powers $h^{k-\ell}$ and $g^{n-m}$ commute, so $h$ and $g$ commute since $G$ is bi-orderable.  Thus $H_1$ is abelian.
\end{proof}

\begin{proof}[Proof of Theorem \ref{main theorem}]  Let $G$ be a bi-orderable group and let $\phi \in \mathrm{Aut}(G)$ be nontrivial.  If $G$ is nonabelian, then by Lemma \ref{no star phis} $\phi$ cannot satisfy (\ref{autom}). By Proposition \ref{main lemma} $\phi$ acts nontrivially on $\mathrm{LO}(G)$, so the action of $\mathrm{Aut}(G)$ on $\mathrm{LO}(G)$ is faithful.

If $G$ is abelian, and if $\phi$ does not satisfy (\ref{autom}), then Proposition \ref{main lemma} tells us that $\phi$ acts nontrivially on $\mathrm{LO}(G)$.  If $\phi$ does satisfy (\ref{autom}), then Lemma \ref{Morris-lemma} tells us that $\phi = \tau_{p/q}$ for some $p/q \in \mathbb{Q}$.  It is easy to see that in this case, $\phi$ acts trivially on $\mathrm{LO}(G)$.  Thus the kernel of the action of $ \mathrm{Aut}(G)$ on $\mathrm{LO}(G)$ consists exactly of the automorphisms $\tau_{p/q}$.
\end{proof}

\section{Non-bi-orderable groups}
\label{nonbo groups}

For certain classes of left-orderable groups, it is sometimes sufficient to examine the action of $\mathrm{Aut}(G)$ on a small subset of $\mathrm{LO}(G)$ (perhaps even a finite subset) in order to determine that the action is faithful.

Recall that a left-ordering of $G$ is \textit{discrete} if there is a smallest positive element.  If $\phi: G\rightarrow G$ is an automorphism, and if $P$ is the positive cone of a discrete left-ordering with smallest positive element $g \in G$, then $\phi(P)$ is the positive cone of a discrete left-ordering whose smallest positive element is $\phi(g)$.  Thus if $g \neq \phi(g)$, then $P \neq \phi(P)$. We apply this idea in the following proposition.

\begin{proposition}
\label{smallest element proposition}
Suppose that $G$ is a left-orderable group with generators $\{ g_i \}_{i \in I}$, and that for each $i \in I$ there exists $P_i \in \mathrm{LO}(G)$ which is the positive cone of a discrete left-ordering with $g_i$ as smallest positive element.  Then $\mathrm{Aut}(G)$ acts faithfully on $\mathrm{LO}(G)$.
\end{proposition}
\begin{proof}
If $\phi : G \rightarrow G$ is a nontrivial automorphism, then there exists a generator $g_i$ such that $\phi(g_i) \neq g_i$.  But then $\phi(P_i) \neq P_i$, so that $\phi$ acts nontrivially on $\mathrm{LO}(G)$.
\end{proof}

\begin{example}
Recall the Artin presentation of braid group $B_n$ is given by
\[ B_n = \left< \sigma_1 , \ldots, \sigma_{n-1}  \hspace{1em}
\begin{array}{|c}
\sigma_i \sigma_j = \sigma_j \sigma_i \mbox{ if $|i-j|>1$}\\
\sigma_i \sigma_j \sigma_i = \sigma_j \sigma_i \sigma_j \mbox{ if $|i-j|=1$} \end{array} 
 \right>.
\]
By Dehornoy, the braid groups $B_n$ are left orderable for all $n$, as is the braid group $B_{\infty}$ \cite{dehornoy94}. The \textit{Dehornoy ordering} of $B_n$ is a left-ordering that is defined in terms of representative words of braids as follows: A word $w$ in the generators $\sigma_1,\dots,\sigma_{n-1}$ is called $i$-positive (respectively $i$-negative) if $w$ contains at least one occurence $\sigma_i$, no occurence of $\sigma_1,\dots ,\sigma_{i-1}$, and every occurence of $\sigma_i$ has positive (respectively negative) exponent. A braid $\beta \in B_n$ is called $i$-positive (respectively $i$-negative) if it admits a representative word $w$ in the generators $\sigma_1,\dots, \sigma_{n-1}$ that is $i$-positive (respectively $i$-negative). The Dehornoy ordering of the braid group $B_n$ is the ordering whose positive cone $P_D$ is the set of all braids $\beta \in B_n$ that are $i$-positive for some $i$.  Using $sh^{n-j} : B_j \rightarrow B_n$ to denote the shift homomorphism sending $\sigma_i$ to $\sigma_{i+j}$, the convex subgroups of $B_n$ are $sh^{n-j}(B_{j}) = \langle \sigma_{n-j+1}, \ldots, \sigma_{n-1} \rangle \subset B_n$ \cite{DDRW08}, in particular the Dehornoy ordering is discrete with smallest positive element $\sigma_{n-1}$.

We can also define a related left-ordering as follows: a word $w$ in generators $\sigma_1,\dots,\sigma_{n-1}$ is called \textit{$i$-reverse positive}, if it has no occurence of $\sigma_{i+1},\dots,\sigma_{n-1}$, and every occurence of $\sigma_i$ has positive exponent. Now similar to Dehornoy ordering, define an ordering $<'_{D}$ on $B_n$, whose positive cone $P'_{D}$ is consists of all braids $\beta \in B_n$ that are $i$-reverse positive for some $i$. 

It a straightforward check that $<'_{D}$ is a also a discrete ordering of $B_n$, with $\sigma_{1}$ as its least positive element. Moreover, the convex subgroups of $B_n$ with respect to $<'_{D}$ are exactly the subgroups $B_j = \langle \sigma_1, \ldots, \sigma_{j-1} \rangle \subset B_n$ for $1\leq j \leq n$.

Now given any $i$ where $1\leq i \leq n-1$, we can construct a left ordering $<_i$ on $B_n$ with $\sigma_i$ as its least positive element. First, we left-order $B_n$ with $<'_{D}$. Since $B_{i+1}$ is convex with respect to $P'_D$, we can replace the left ordering $<'_D$ on $B_{i+1}$ with the left ordering of $<_D$. Denote the resulting ordering of $B_n$ by $<_i$. By construction, $<_i$ is a discrete ordering with $\sigma_i$ as its least positive element. Based on this construction and Proposition \ref{smallest element proposition}, $\mathrm{Aut}(B_n)$ acts faithfully on $\mathrm{LO}(B_n)$.  

This same construction can also be used to produce a left-ordering of $B_{\infty}$ with $\sigma_i$ as smallest positive element for all $i \geq 1$.  Thus $\mathrm{Aut}(B_{\infty})$ acts faithfully on $\mathrm{LO}(B_{\infty})$ as well.
\qed
\end{example}

If $K$ and $H$ are bi-orderable groups and 
\[ 1 \rightarrow K \rightarrow G \rightarrow H \rightarrow 1
\]
is a short exact sequence, then $G$ can be lexicographically bi-ordered if and only if there exists a bi-ordering of $K$ whose positive cone is invariant under the conjugation action of $H$.  By relaxing this condition, we are able to create groups which are \textit{not} bi-orderable, but for which the automorphism group acts faithfully on the space of left-orderings.

\begin{theorem}
\label{extension theorem}
Suppose that $G$ is left-orderable and that
\[ 1 \rightarrow K \rightarrow G \rightarrow \mathbb{Z} \rightarrow 1
\]
is a short exact sequence of groups.  Suppose that $\mathrm{Aut}(K)$ acts faithfully on $\mathrm{LO}(K)$.  If conjugation by the generator of $\mathbb{Z}$ preserves a left-ordering of $K$, then $\mathrm{Aut}(G)$ acts faithfully on $\mathrm{LO}(G)$.
\end{theorem}
\begin{proof}
Suppose that $\phi : G \rightarrow G$ is a nontrivial automorphism.  If $\phi(K) \neq K$, choose $g \in K$ with $\phi(g) \notin K$.  Then by choosing signs appropriately, we may use the given short exact sequence to construct a positive cone  $P \subset G$ for which $g \in P$ while $\phi(g) \notin P$.  Thus $\phi(P) \neq  P$. 

On the other hand, suppose that $\phi(K) =K$.  If there exists $k \in K$ for which $\phi(k) \neq k$, then we know there is a positive cone $P_K \in \mathrm{LO}(K)$ for which $\phi(P_K) \neq P_K$ since  $\mathrm{Aut}(K)$ acts faithfully on $\mathrm{LO}(K)$.  Using the given short exact sequence we may extend $P_K$ to a positive cone $P \subset G$ satisfying $\phi(P) \neq P$.

Last, suppose that $\phi(k) =k$ for all $k \in K$, and choose $t \in G$ which maps to the generator of $\mathbb{Z}$.  Equip $K$ with a positive cone $P_K$ that is preserved by conjugation by $t$, and proceed as in \cite[Lemma 3.4]{LRR09}.  Note that every $g \in G$ can be written uniquely as $kt^n$ for some $n \in \mathbb{Z}$ and $k \in K$, and since $\phi$ is nontrivial and satisfies $\phi(k) = k$ for all $k \in K$ it follows that $\phi(t) \neq t$.  Construct a positive cone $P \subset G$ as follows:  an element $kt^n$ is in $P$ if $k \in P_K$ or $k=1$ and $n>0$.  Then $P$ clearly satisfies $P \cup P^{-1} = G \setminus \{ 1\}$ and $P \cap P^{-1} = \emptyset$.  Moreover if $kt^n$ and $k't^m$ are both in $P$, then so is $kt^nk't^m = k(t^nk't^{-n})t^{m+n}$ since conjugation by $t$ preserves $P_K$.  One can easily verify that the subgroup $\langle t \rangle$ is convex relative to the ordering of $G$ determined by $P$, so that $P$ determines a discrete ordering of $G$ with $t$ as smallest positive element.  The positive cone $\phi(P)$ will determine a left-ordering of $G$ with $\phi(t)$ as smallest positive element.  As $\phi(t) \neq t$, we conclude that $\phi(P) \neq P$.
\end{proof}

If $K$ is a bi-orderable group, automorphisms $\phi:K \rightarrow K$  which preserve a left-ordering of $K$ but not a bi-ordering are likely quite common.  However, there is little in the literature dealing with automorphism-invariant left-orderings, as the focus has primarily been on automorphism-invariant bi-orderings \cite{PR03, PR06, LRR08}.   

Here is one example of how an automorphism-invariant left-ordering (which is not a bi-ordering) may arise, which we use to illustrate an application of Theorem \ref{extension theorem}.

\begin{example}
Set $K = \mathbb{Q}^2 \rtimes \mathbb{Z}$ where the conjugation action of $\mathbb{Z}$ on $\mathbb{Q}^2$ is by the matrix $A = \left( \begin{smallmatrix} 1&2\\ 1&1 \end{smallmatrix} \right)$.  Then $K$ is bi-orderable, since the action of $A$ preserves the bi-ordering of $\mathbb{Z}^2$ defined by $(a,b) > (0,0)$ if and only if $(a,b)\cdot (\sqrt{2}, 1) = b+\sqrt{2} a>0$.  In fact, since the eigenvectors of $A$ are $\left( \begin{smallmatrix} \sqrt{2}\\ 1 \end{smallmatrix} \right)$ and $\left( \begin{smallmatrix} -\sqrt{2}\\ 1 \end{smallmatrix} \right)$ with positive and negative eigenvalues respectively, the ordering described above (and its opposite) are the only orderings of $\mathbb{Q}^2$ preserved by $A$.  Thus $K$ is bi-orderable and nonabelian, so by Theorem \ref{main theorem} $\mathrm{Aut}(K)$ acts faithfully on $\mathrm{LO}(K)$. 

Now we define $G = K \rtimes \mathbb{Z}$ where the action of the generator of $\mathbb{Z}$ on an element of $K$ is $((a,b), c) \mapsto (-A(a,b)^T,c)$.  The action of $-A$ on the subgroup $\mathbb{Q} ^2 \subset K$, having the same eigenvectors as $A$ but with eigenvalues of opposite sign, preserves only the ordering defined by $(a,b) > (0,0)$ if and only if $(a,b)\cdot (-\sqrt{2}, 1) = b -\sqrt{2} a>0$, and its opposite.   Using this ordering on $\mathbb{Q}^2$, and lexicographically ordering $K$ using the short exact sequence $ 1 \rightarrow \mathbb{Q}^2 \rightarrow K \rightarrow \mathbb{Z} \rightarrow 1$, we arrive at a left-ordering of $K$ preserved by the action of the generator of $\mathbb{Z}$.  

We conclude $\mathrm{Aut}(G)$ will act faithfully on $\mathrm{LO}(G)$ by Theorem \ref{extension theorem}.

Note that $G$ is left-orderable by a straightforward short exact sequence argument, but is not bi-orderable since the actions of $A$ and $-A$ on $\mathbb{Q}^2 \subset K$ do not preserve a common ordering, so Theorem \ref{main theorem} does not apply.  Proposition \ref{smallest element proposition} also cannot apply to $G$ since any generator of $\mathbb{Q}^2 \subset G$ cannot be the smallest positive element of a left-ordering of $G$.
 \qed
\end{example}

Despite these extensions and examples, one cannot hope to replace ``bi-orderable" in Theorem \ref{main theorem} with either the weaker condition of local indicability or the condition that $G$ admit an ordering that is recurrent for every cyclic subgroup (See \cite{Morris06} for more information on recurrent orderings).  Koberda points out that for the Klein bottle group, $K = \langle x,y \mid xyx^{-1} = y^{-1} \rangle$, the action of $\mathrm{Aut}(K)$ on $\mathrm{LO}(K)$ is not faithful.  Yet $K$ is both locally indicable and admits recurrent orderings, as it only has four left-orderings.

\section{Applications and generalizations}
\label{VLOG}

The action of $\mathrm{Aut}(G)$ on $\mathrm{LO}(G)$ is connected to the action of $\mathrm{Aut}(G)$ on $\partial G$ by the following theorem.  Though not stated in full generality in \cite{Koberda11}, the proof below appears there as part of the proof of \cite[Theorem 1.2]{Koberda11}.  As it is relatively short, we repeat it here for the reader's convenience.  For background and futher information on hyperbolic groups, see \cite{Koberda11, Gromov87, KB02}.

\begin{proposition}
\label{Koberda prop}
If $G$ is a left-orderable hyperbolic group and $\mathrm{Aut}(G)$ acts faithfully on $\mathrm{LO}(G)$, then it acts faithfully on $\partial G$.
\end{proposition}
\begin{proof}
Recall that for each element $g \in G$, there are two distinct points in the boundary $\partial G$ defined by $x_g = \lim_{n \to \infty} g^n$ and $y_g = \lim_{n \to \infty} g^{-n}$.  Moreover, given $g, h \in G$ if $\langle g, h \rangle$ is not a virtually cyclic group, then $g$ and $h$ determine distinct points on the boundary.

Choose a nontrivial automorphism $\phi \in \mathrm{Aut}(G)$, $g \in G$ and $P \in \mathrm{LO}(G)$ such that $g \in P$ and $\phi(g) \notin P$ (and thus $\phi(P) \neq P$).  Since we cannot have $g^{k} = \phi(g)^{\ell}$ for some $k,\ell >0$, there are two cases.  Recall we defined $I(g)$ in Section \ref{main section} to be the isolator of the cyclic subgroup $\langle g \rangle$.

\noindent \textbf{Case 1.} $\phi(g) \in I(g)$ and there exists $k,\ell >0$ such that $g^{-k} = \phi(g)^{\ell}$.  In this case, observe that $x_g = \lim_{n \to \infty} (g^{\ell})^n$, so that $\phi(x_g) = \lim_{n \to \infty} \phi(g^{\ell})^n = \lim_{n \to \infty}  (g^{-k})^n = y_g$, so that $\phi$ acts nontrivially on $\partial G$.

\noindent \textbf{Case 2.} $\phi(g) \notin I(g)$.  Then $\phi(g)$ and $g$ do not generate a virtually cyclic subgroup, so $x_g$ and $\phi(x_g) = x_{\phi(g)}$ are distinct.  Thus $\phi$ acts nontrivially on $\partial G$.
\end{proof}

Consequently, by applying Theorem \ref{main theorem}, we arrive at Corollary \ref{hyperbolic}.  If $G$ is hyperbolic and satisfies the hypotheses of Theorem \ref{extension theorem} or Proposition \ref{smallest element proposition}, then $\mathrm{Aut}(G)$ acts faithfully on $\partial G$, too.  However it seems difficult to construct a hyperbolic group $G$ satisfying the hypotheses of either result.

There are also two natural generalizations one may consider, both developed by Witte Morris in \cite{Morris12}.  First, one may replace the automorphism group with the commensurator group $\mathrm{Comm}(G)$ of $G$.  Recall that a \textit{commensuration} of a group $G$ is an isomorphism $\phi :H_1 \rightarrow H_2$ of finite index subgroups $H_i \subset G$.  Two commensurations $\phi :H_1 \rightarrow H_2$ and $\phi' :H_1' \rightarrow H_2'$ are equivalent if there exists a finite index subgroup $H \subset H_1 \cap H_1'$ such that  $\phi |_H = \phi' |_{H}$.  The set of equivalence classes of commensurations forms the \textit{commensurator group} $\mathrm{Comm}(G)$ of $G$.

Witte Morris points out that for torsion free locally nilpotent groups, $\mathrm{Comm}(G)$ acts naturally on $\mathrm{LO}(G)$.   This follows from an application of Koberda's theorem (Theorem \ref{koberdas theorem}), and the fact that for every subgroup $H$ of a torsion-free locally nilpotent group $G$, the restriction map $r: \mathrm{LO}(G) \rightarrow \mathrm{LO}(H)$ is surjective.  When $G$ is a bi-orderable group, the restriction $r: \mathrm{LO}(G) \rightarrow \mathrm{LO}(H)$ is not a surjective map in general, so this generalization is not possible in our setting.

However, using the restriction map $r: \mathrm{LO}(G) \rightarrow \mathrm{LO}(H)$ for each finite index subgroup $H \subset G$, one can define the space of \textit{virtual} left-orderings of $G$ as the limit
\[ \mathrm{VLO}(G) = \varinjlim \mathrm{LO}(H),
\]
where the limit is over all finite-index subgroups $H$ of $G$ \cite{Morris12}.  When $P \in \mathrm{LO}(H)$ and $H$ is a finite index subgroup of $G$, we will denote the corresponding element of $\mathrm{VLO}(G)$ by $[P]$.  Then $\mathrm{Comm}(G)$ naturally acts on $\mathrm{VLO}(G)$: for each commensuration $\phi :H_1 \rightarrow H_2$ and each positive cone $P \in \mathrm{LO}(H)$, set $\phi([P]) = [\phi(P \cap H_1)]$.  It is straightforward to check that this definition respects the necessary equivalence relations.

\begin{lemma}
\label{VLOG action lemma}
Let $G$ be a left-orderable group and  $\phi :H_1 \rightarrow H_2$ a commensuration of $G$ where $H_1$ is abelian.  If $\phi = \tau_{p/q}$ for some $p/q \in \mathbb{Q}$ then the element of $\mathrm{Comm}(G)$ represented by $\phi$ acts trivially on $\mathrm{VLO}(G)$.
\end{lemma}
\begin{proof}
Suppose that $H$ is a finite index subgroup and $P \subset H$ is the positive cone of a left-ordering.  Consider $P \cap H_1$ and $ \phi(P \cap H_1)$.   The first is the positive cone of a left-ordering of $H \cap H_1$, the second is the positive cone of a left-ordering of $H \cap H_2$.  Using the fact that $\phi$ satisfies (\ref{autom}), one can show that these orderings agree on the finite index subgroup $H \cap H_1 \cap H_2$ so that $[P]=[\phi(P \cap H_1)]$, and thus $\phi$ acts trivially on $\mathrm{VLO}(G)$.
\end{proof}

\begin{theorem}  Let $G$ be a bi-orderable group. 

\noindent (i) If $G$ is not virtually abelian then $\mathrm{Comm}(G)$ acts faithfully on $\mathrm{VLO}(G)$.

\noindent (ii) If $G$ is virtually abelian then the kernel of the action of $\mathrm{Comm}(G)$ on $\mathrm{VLO}(G)$ contains precisely the elements represented by commensurations $\tau_{p/q} : H_1 \rightarrow H_2$, if any such commensurations exist.

\end{theorem}
\begin{proof}
First suppose that $G$ is not virtually abelian, and let $\phi: H_1 \rightarrow H_2$ be a nontrivial commensuration of $G$.  By Lemma \ref{no star phis}, $\phi$ cannot satisfy (\ref{autom}) since $H_1$ is not abelian.  Thus there exists $g \in H_1$ such that $\phi(g) \neq g$ and either $\phi(g)^n = g^{-m}$ for some $m,n >0$ or $I(g) \neq I(\phi(g))$.  In either case we can construct a left-ordering of $G$ with positive cone $P$ satisfying $g \in P$ and $\phi(g) \notin P$ using arguments identical to those in the proof of Lemma \ref{main lemma}.  Then $[P] \neq [\phi(P \cap H_1)]$, so (the class of) $\phi: H_1 \rightarrow H_2$ acts nontrivially on $\mathrm{VLO}(G)$.

On the other hand, suppose $G$ is virtually abelian, and let $\phi: H_1 \rightarrow H_2$ be a nontrivial commensuration of $G$.  If $\phi$ does not satisfy (\ref{autom}), then an argument identical to the previous paragraph shows that the class of $\phi$ acts nontrivially on $\mathrm{VLO}(G)$.  On the other hand, if $\phi$ does satisfy (\ref{autom}), then $\phi = \tau_{p/q}$ for some $p/q \in \mathbb{Q}$ by Lemma \ref{Morris-lemma}.  In this case, $\phi = \tau_{p/q}$ acts trivially on $\mathrm{VLO}(G)$ by Lemma \ref{VLOG action lemma}.
\end{proof}

\bibliographystyle{plain}

\bibliography{ordbook}

\end{document}